\documentclass[a4paper, leqno, 10pt]{amsart}
\usepackage[a4paper,top=1.1in,bottom=1.1in,left=1.1in,right=1.1in, bindingoffset=0mm]{geometry}

\usepackage[T1]{fontenc}	
\usepackage{graphicx}
\usepackage{amssymb,centernot}
\usepackage{upgreek}
\usepackage{mathrsfs}
\usepackage[usenames,dvipsnames]{xcolor}
\usepackage[inline,shortlabels]{enumitem}
\usepackage[sort,numbers]{natbib}						
\usepackage{verbatim}								
\usepackage{dsfont}									
\usepackage{microtype}

\usepackage{pdflscape}
\usepackage{afterpage}

\usepackage{faktor}
\usepackage{anyfontsize}
\usepackage[
						colorlinks,				
						breaklinks,unicode,
						hypertexnames=false,
						citecolor=OliveGreen,
						linkcolor=Maroon
						]{hyperref} 

\usepackage{multirow}

\usepackage{xypic}
\usepackage{mathtools}
\mathtoolsset{showonlyrefs}

\usepackage{xspace}

\usepackage{booktabs}

\usepackage{tikz}

\usetikzlibrary{calc, positioning, shapes.geometric, patterns, cd}

\newlength{\hatchspread}
\newlength{\hatchthickness}
\newlength{\hatchshift}
\newcommand{\hatchcolor}{}
\tikzset{hatchspread/.code={\setlength{\hatchspread}{#1}},
         hatchthickness/.code={\setlength{\hatchthickness}{#1}},
         hatchshift/.code={\setlength{\hatchshift}{#1}},
         hatchcolor/.code={\renewcommand{\hatchcolor}{#1}}}
\tikzset{hatchspread=3pt,
         hatchthickness=0.4pt,
         hatchshift=0pt,
         hatchcolor=black}
\pgfdeclarepatternformonly[\hatchspread,\hatchthickness,\hatchshift,\hatchcolor]
   {custom north west lines}
   {\pgfqpoint{\dimexpr-2\hatchthickness}{\dimexpr-2\hatchthickness}}
   {\pgfqpoint{\dimexpr\hatchspread+2\hatchthickness}{\dimexpr\hatchspread+2\hatchthickness}}
   {\pgfqpoint{\dimexpr\hatchspread}{\dimexpr\hatchspread}}
   {
    \pgfsetlinewidth{\hatchthickness}
    \pgfpathmoveto{\pgfqpoint{0pt}{\dimexpr\hatchspread+\hatchshift}}
    \pgfpathlineto{\pgfqpoint{\dimexpr\hatchspread+0.15pt+\hatchshift}{-0.15pt}}
    \ifdim \hatchshift > 0pt
      \pgfpathmoveto{\pgfqpoint{0pt}{\hatchshift}}
      \pgfpathlineto{\pgfqpoint{\dimexpr0.15pt+\hatchshift}{-0.15pt}}
    \fi
    \pgfsetstrokecolor{\hatchcolor}
    \pgfusepath{stroke}
   }

\pgfdeclarepatternformonly[\hatchspread,\hatchthickness,\hatchshift,\hatchcolor]
   {custom north east lines}
   {\pgfqpoint{\dimexpr-2\hatchthickness}{\dimexpr-2\hatchthickness}}
   {\pgfqpoint{\dimexpr\hatchspread+2\hatchthickness}{\dimexpr\hatchspread+2\hatchthickness}}
   {\pgfqpoint{\dimexpr\hatchspread}{\dimexpr\hatchspread}}
   {
    \pgfsetlinewidth{\hatchthickness}
    \pgfpathmoveto{\pgfqpoint{\dimexpr\hatchshift-0.15pt}{-0.15pt}}
    \pgfpathlineto{\pgfqpoint{\dimexpr\hatchspread+0.15pt}{\dimexpr\hatchspread-\hatchshift+0.15pt}}
    \ifdim \hatchshift > 0pt
      \pgfpathmoveto{\pgfqpoint{-0.15pt}{\dimexpr\hatchspread-\hatchshift-0.15pt}}
      \pgfpathlineto{\pgfqpoint{\dimexpr\hatchshift+0.15pt}{\dimexpr\hatchspread+0.15pt}}
    \fi
    \pgfsetstrokecolor{\hatchcolor}
    \pgfusepath{stroke}
   }

\makeatletter
\newcommand*{\centerfloat}{%
  \parindent \z@
  \leftskip \z@ \@plus 1fil \@minus \textwidth
  \rightskip\leftskip
  \parfillskip \z@skip}
\makeatother

\usepackage{xparse}

\ExplSyntaxOn
\NewDocumentCommand{\makeabbrev}{mmm}
 {
  \yoruk_makeabbrev:nnn { #1 } { #2 } { #3 }
 }

\cs_new_protected:Npn \yoruk_makeabbrev:nnn #1 #2 #3
 {
  \clist_map_inline:nn { #3 }
   {
    \cs_new_protected:cpn { #2 } { #1 { ##1 } }
   }
 }
 \ExplSyntaxOff

\makeabbrev{\textbf}{tbf#1}{a,b,c,d,e,f,g,h,i,j,k,l,m,n,o,p,q,r,s,t,u,v,w,x,y,z,A,B,C,D,E,F,G,H,I,J,K,L,M,N,O,P,Q,R,S,T,U,V,W,X,Y,Z}

\makeabbrev{\textbf}{bf#1}{a,b,c,d,e,f,g,h,i,j,k,l,m,n,o,p,q,r,s,t,u,v,w,x,y,z,A,B,C,D,E,F,G,H,I,J,K,L,M,N,O,P,Q,R,S,T,U,V,W,X,Y,Z}

\makeabbrev{\textsf}{tsf#1}{a,b,c,d,e,f,g,h,i,j,k,l,m,n,o,p,q,r,s,t,u,v,w,x,y,z,A,B,C,D,E,F,G,H,I,J,K,L,M,N,O,P,Q,R,S,T,U,V,W,X,Y,Z}

\makeabbrev{\mathsf}{mss#1}{a,b,c,d,e,f,g,h,i,j,k,l,m,n,o,p,q,r,s,t,u,v,w,x,y,z,A,B,C,D,E,F,G,H,I,J,K,L,M,N,O,P,Q,R,S,T,U,V,W,X,Y,Z}

\makeabbrev{\mathfrak}{mf#1}{a,b,c,d,e,f,g,h,i,j,k,l,m,n,o,p,q,r,s,t,u,v,w,x,y,z,A,B,C,D,E,F,G,H,I,J,K,L,M,N,O,P,Q,R,S,T,U,V,W,X,Y,Z}

\makeabbrev{\mathrm}{mrm#1}{a,b,c,d,e,f,g,h,i,j,k,l,m,n,o,p,q,r,s,t,u,v,w,x,y,z,A,B,C,D,E,F,G,H,I,J,K,L,M,N,O,P,Q,R,S,T,U,V,W,X,Y,Z}

\makeabbrev{\mathbf}{mbf#1}{a,b,c,d,e,f,g,h,i,j,k,l,m,n,o,p,q,r,s,t,u,v,w,x,y,z,A,B,C,D,E,F,G,H,I,J,K,L,M,N,O,P,Q,R,S,T,U,V,W,X,Y,Z}

\makeabbrev{\mathcal}{mc#1}{A,B,C,D,E,F,G,H,I,J,K,L,M,N,O,P,Q,R,S,T,U,V,W,X,Y,Z}

\makeabbrev{\mathbb}{mbb#1}{A,B,C,D,E,F,G,H,I,J,K,L,M,N,O,P,Q,R,S,T,U,V,W,X,Y,Z}

\makeabbrev{\mathscr}{ms#1}{A,B,C,D,E,F,G,H,I,J,K,L,M,N,O,P,Q,R,S,T,U,V,W,X,Y,Z}

\makeabbrev{\mathrm}{#1}{
Id,id,ran,rk,diag,stab,ann,conv,pr,ev,tr,End,Hom,sgn,im,op,can,fin,ext,red,tot,
rot,usc,lsc,Lip,LocLip,lip,bSymLip,osc,AC,loc,uloc,spec,coz,z,ul,
supp,Opt,Adm,Cpl,Geo,GeoSel,GeoOpt,GeoAdm,GeoCpl,reg,
bd,co,Ric,Exp,dExp,dist,seg,Seg,cut,fcut,Cut,SDiff,Iso,Isom,diam,cl,Homeo,Diff,Der,vol,dvol,inj,relint, Graph, sub,codim,
var,law,Var,Poi,Gam,pa,so,iso,fs,inv,pqi,mix,
TestF,
}

\makeabbrev{\mathsf}{#1}{DP,CD,BE,MCP,wMTW,MTW,RCD,ncRCD,QCD,EVI,Irr,IH,SC,wFe,VA,UP,Curv,Alex,CAT}

\newcommand{\dom}[1]{\msD(#1)}

\let\epsilon\varepsilon

\let\temp\phi
\let\phi\varphi
\let\varphi\temp

\newcommand{\diff}{\mathop{}\!\mathrm{d}}

\DeclareSymbolFont{symbolsC}{U}{pxsyc}{m}{n}
\SetSymbolFont{symbolsC}{bold}{U}{pxsyc}{bx}{n}
\DeclareFontSubstitution{U}{pxsyc}{m}{n}
\DeclareMathSymbol{\medcirc}{\mathbin}{symbolsC}{7}
\DeclareSymbolFont{symbolsZ}{OMS}{pxsy}{m}{n}
\SetSymbolFont{symbolsZ}{bold}{OMS}{pxsy}{bx}{n}
\DeclareFontSubstitution{OMS}{pxsy}{m}{n}

\newcommand{\N}{{\mathbb N}}
\newcommand{\R}{{\mathbb R}}
\newcommand{\C}{{\mathbb C}}

\allowdisplaybreaks

\usetikzlibrary{shapes.misc}
\tikzset{cross/.style={cross out, draw=black, minimum size=2*(#1-\pgflinewidth), inner sep=0pt, outer sep=0pt},
cross/.default={4pt}}

\newcommand{\comma}{\,\,\mathrm{,}\;\,}

\newcommand{\fstop}{\,\,\mathrm{.}}

\renewcommand{\iint}{\int\!\!\!\!\int}

\usepackage{scrextend}						

\newcommand{\D}{\mcD} 

\newcommand{\QP}{{\mu}}

\newcommand{\dUpsilon}{{\boldsymbol\Upsilon}}

\newcommand{\cquad}{\comma \quad}

\newcommand{\U}{\dUpsilon}
\newcommand{\E}{\mathcal E}

\renewcommand{\1}{\mathbf 1}

\newcommand{\sine}{\mathsf{sine}}

\numberwithin{equation}{section}
\theoremstyle{plain}	
\newtheorem{thm}{Theorem}[section]
\newtheorem*{thm*}{Theorem A}
\newtheorem*{thm**}{Theorem B}
\newtheorem*{mthm*}{Main Theorem}
\newtheorem{theorem}{Theorem}

\newtheorem{lemma}[thm]{Lemma}
\newtheorem*{cor*}{Corollary A}
\newtheorem*{cor**}{Corollary B}
\newtheorem*{cor***}{Corollary C}

\theoremstyle{definition}
\newtheorem*{defs*}{Definition}
\theoremstyle{remark}

\newtheorem{remark}[thm]{\bf Remark}

\newtheorem*{asm*}{{\bf List of Assumptions}}

\renewcommand{\paragraph}[1]{\medskip\emph{#1}.}
\renewcommand{\#}{\sharp}

\newcommand{\anti}{\mathrm{anti}}

\newcommand{\Tr}{\mathsf{Tr}}

\makeatletter
\@namedef{subjclassname@2020}{%
  \textup{2020} Mathematics Subject Classification}
\makeatother

\begin{document}
\title[Spectral gap for unlabelled Ginibre Interacting Brownian motion]{Spectral gap for unlabelled Ginibre Interacting Brownian motion}

\author[K.~Suzuki]{Kohei Suzuki}
\address{Department of Mathematical Science, Durham University, South Road, DH1 3LE, United Kingdom}
\email{kohei.suzuki@durham.ac.uk}
\thanks{\hspace{-5.5mm}  Department of Mathematical Science, Durham University, South Road, DH1 3LE, United Kingdom
\\
\hspace{2.0mm} E-mail: kohei.suzuki@durham.ac.uk
}

%

\renewcommand{\abstractname}{\normalsize  Abstract}
\begin{abstract} \normalsize
We determine the exact spectral gaps for both the finite-particle and infinite-particle unlabelled Ginibre interacting Brownian motions. In every finite-particle system, the spectral gap is equal to one, whereas for the infinite-particle system it is equal to two. The infinite-particle dynamics therefore has a strictly larger spectral gap than any of its finite-particle counterparts, revealing an enhancement of spectral gaps in the passage from finite to infinite particle systems.
\end{abstract}

\maketitle

\vspace{-7mm}
\section{Introduction}
The unlabelled Ginibre interacting Brownian motion is a system of interacting Brownian motions whose equilibrium measure is  the Ginibre point process.  The latter is one of the most fundamental determinantal point processes in random matrix theory: it describes the local bulk distribution of eigenvalues of random matrices with independent complex Gaussian entries and, at the same time, the Gibbs measure of a two-dimensional logarithmic Coulomb gas at inverse temperature~$\beta=2$.  Its associated diffusion provides a   model of two-dimensional Brownian particles interacting through a singular and long-range logarithmic repulsion.  It therefore lies at the intersection of random matrix theory, interacting particle systems, determinantal point processes, and  stochastic differential equations.

\smallskip
The main purpose of this paper is to determine the exact spectral gaps for both the finite-particle and infinite-particle unlabelled Ginibre interacting Brownian motions.  In particular, we prove  that the infinite-particle dynamics has a spectral gap and show that its exact value is~two.  From a dynamical viewpoint, this yields an exponential convergence to the equilibrium in $L^2(\mu)$ for the infinite-particle dynamics.

There are three aspects of the infinite-particle result that are worth emphasising. First, the infinite-particle system is translation-invariant. In many translation-invariant particle systems, long-wavelength density fluctuations produce arbitrarily slow modes and prevent the existence of a spectral gap; this is the case, for example, for the $\sine_2$ dynamics, namely the infinite-particle unlabelled Dyson Brownian motion at inverse temperature~$\beta=2$; see~\cite{Suz24}. Thus, the existence of the spectral gap is somewhat unexpected.  

Second,  the planar logarithmic interaction is not convex, so the result does not follow from the standard Bakry--\'Emery curvature method. The proof exploits the holomorphic structure of the Ginibre ensemble, together with a quantitative projection estimate showing that the holomorphic component of every fixed local observable vanishes in the infinite-particle limit.

Third, the spectral gap of the infinite-particle system is strictly larger than that of every finite-particle counterpart: the spectral gap of the finite-particle is exactly~one, independently of the number of particles, whereas the spectral gap of the infinite-particle  is~two. We refer to this phenomenon as {\it an enhancement of the spectral gap}.  As explained below, the phenomenon is caused by the disappearance, in the infinite-particle limit, of the global centre-of-mass mode that realises the exact spectral gap in the finite-particle case.

\subsection{Main result}
The {\it Ginibre point process} is the determinantal point process in the complex plane~$\C$ whose correlation kernel density  is given by
$$
  K(z,w)
  =\frac1\pi \exp\left\{-\frac{|z|^2}{2}-\frac{|w|^2}{2}+z\overline w\right\} \comma
$$
where $\overline w$ is the complex conjugate of $w$.
Let $\mu$ denote the law of the Ginibre point process, which is a probability measure on the configuration space over $\C$
$$\U:=\biggl\{\gamma =\sum_{i=1}^N \delta_{x_i}: N \in \N_0 \cup \{+\infty\},\ x_i \in \C,\ \gamma(K)<\infty \ \text{for all compact $K \subset \C$}\biggr\} \comma$$
where $\delta_x$ denotes the Dirac measure at $x$. 
We consider the pre-Dirichlet form associated with $\mu$ given as 
$$
  \mathcal E(F)
  =\frac12\int_{\U} \sum_{x\in \gamma} |\nabla_x F(\gamma)|^2\,\mu(\diff \gamma) \cquad \mathcal E(F, G)= \frac{1}{4}\Bigl( \E(F+G) - \E(F-G) \Bigr) \comma 
$$
defined for cylinder functions, i.e., functions $F: \U\to \R$ having the following forms:
$$
  F(\gamma)=\Theta\bigl(\langle f_1,\gamma\rangle,\ldots,
                   \langle f_m,\gamma\rangle\bigr) \cquad m \in \N \comma
$$
where $f_1,\ldots,f_m\in C_c^\infty(\C \to \R)$ are real-valued compactly supported smooth functions in~$\C$, $\langle f,\gamma\rangle = \int_\C f \diff \gamma$ and
$\Theta\in C_b^\infty(\R^m)$.  
Here
$\nabla_x$ is the gradient operator: for cylinder functions, it is  given as 
$$
  \nabla_xF(\gamma)
  =\sum_{i=1}^m
  (\partial_i\Theta)\bigl(\langle f_1,\gamma\rangle,\ldots,
                      \langle f_m,\gamma\rangle\bigr)
  \nabla f_i(x) \comma
$$
where $\nabla$ is the real two-dimensional gradient in the variable $x$.
By~\cite[Lemma 2.1, Theorem 2.3 and Corollary 2.3]{Osa13}, $\E$ is closable with respect to $\|\cdot\|_{L^2(\mu)}^2 + \E(\cdot)$ and the closure $(\E, \mathcal D)$ is a quasi-regular strongly local $\mu$-symmetric Dirichlet form.  Note that due to~\cite[Theorem~3.1]{OsaTan14b},  the Dirichlet form defined as the closure on cylinder functions is the same as the one defined on smooth local functions in the sense of \cite[p.5]{Osa13}.

By a fundamental theorem of the Dirichlet form theory (e.g., \cite[Theorem 1.5.2]{CheFuk11}),  $(\E, \mathcal D)$ therefore admits an associated
$\mu$-reversible diffusion process in $\U$ whose transition semigroup $u_t=T_tf$ acting on~$L^2(\U, \QP)$ is given as the unique solution to
$$\partial_t u_t = \mathcal A u_t \cquad u_0=f\comma$$
where $(\mathcal A, \dom{\mathcal A})$ is the infinitesimal generator, i.e., the unique self-adjoint operator satisfying $\E(F,G)=-\langle \mathcal AF, G \rangle_{L^2(\mu)} $ for $F \in \dom{\mathcal A}$ and $G \in \D$.
Let $(\Xi_t)_{t\geq0}$ denote the associated diffusion process in~$\U$,  called {\it infinite-dimensional unlabelled Ginibre interacting Brownian motion}. 
Writing a labelled version as
$$
  \Xi_t=\sum_{i=1}^{\infty}\delta_{X_t^i} \comma
$$
one obtains, for $\mu$-almost every initial configuration,
the infinite-dimensional stochastic differential equation
$$
  \diff X_t^i
  =
  \diff B_t^i
  +
  \lim_{r\to\infty}
  \sum_{\substack{j\neq i\\ |X_t^i-X_t^j|<r}}
  \frac{X_t^i-X_t^j}{|X_t^i-X_t^j|^2} \diff t
  \cquad i\in\mathbb N \comma
$$
where $\{B^i\}_{i\in\mathbb N}$ are independent standard Brownian
motions in $\mathbb R^2\simeq\mathbb C$. 
The limits are
defined through the logarithmic derivative of $\mu$ and
converge locally in $L^p$ with respect to the one-Campbell measure for
every $1\le p<2$. The existence of the associated labelled solution for $\mu$-a.e.~initial configuration
 has been proved in
\cite{Osa12} and \cite[Lemma~7.4 and Theorem~7.1]{OsaTan20}.

\smallskip
Our main result is to give the exact spectral gap of $(\E, \mathcal D)$. Let $\Var(F)=\int_\U F^2 \diff \mu - (\int_\U F \diff \mu)^2$.
\begin{theorem}[Spectral gap for infinite-system]
\label{thm:IG} 
$$
\inf_{\substack{F \in \mathcal D \\  \Var(F) > 0}} \frac{\E(F)}{\Var(F)} = 2 \fstop
$$
\end{theorem}

\subsection{Finite particle case}
Define the probability measure on $\C^N$ as 
$$
  \mu_N(\diff z_1\cdots \diff z_N)
  =\frac1{Z_N}|\Delta_N(z_1,\ldots,z_N)|^2 e^{-\Phi_N(z_1,\ldots,z_N)} \diff z_1\cdots \diff  z_N\comma
$$
where $\Delta_N(z_1,\ldots,z_N):=\prod_{1 \le i<j \le N}(z_i-z_j)$, $\Phi_N(z_1,\ldots,z_N)=\sum_{i=1}^N |z_i|^2$  and $Z_N$ is the normalising constant. Let $\U^{(N)}=\{\gamma \in \U: \gamma(\C)=N\}$ be the $N$-point configuration space and $\pi_N: \C^N \to \U^{(N)}$ be the symmetric quotient map~$(z_1, \ldots, z_N) \mapsto \pi_N(z_1, \ldots, z_N)=\sum_{i=1}^N \delta_{z_i}$ . The push-forward measure~$\mu^{[N]}=(\pi_N)_\# \mu_N$  is the law of the $N$-point determinantal point process in~$\C$ with correlation~kernel density
$$
  K_N(z,w)
  =\frac1\pi
  \exp\left\{-\frac{|z|^2}{2}-\frac{|w|^2}{2}\right\}
  \sum_{k=0}^{N-1}\frac{(z\overline w)^k}{k!} \fstop
$$
We consider the Dirichlet form
$$
  \mathcal E_N(F)
  =\frac12\int_{\C^N}\sum_{i=1}^N |\nabla_i F|^2 \diff \mu_N \comma
$$
where $\nabla_i$ is the real two-dimensional gradient in the variable $z_i$. We write $\mathcal D_N$ for the closure of the subspace $C_{c, \mathrm{sym}}^\infty(\C^N \to \R) \subset C_{c}^\infty(\C^N \to \R)$ of symmetric smooth functions with compact support with respect to $\|\cdot\|_{L^2(\mu_N)}^2 + \E_N(\cdot)$.
 Let $\Var_N(F)=\int_{\C^N} F^2 \diff \mu_N - (\int_{\C^N} F \diff \mu_N)^2$.
\begin{theorem}[Spectral gap for $N$-system] For every $N \ge 1$
\label{prop:FG}
$$
\inf_{\substack{F \in \mathcal D_N \\  \Var_N(F) > 0}} \frac{\E_N(F)}{\Var_N(F)} = 1 \fstop
$$
\end{theorem}
We remark that Theorem~\ref{prop:FG} claims the uniform spectral gap only for symmetric functions~$\mathcal D_N$. For general non-symmetric functions, the spectral gap is indeed not uniformly bounded in $N$, see Remark~\ref{rem:SDM}. 

\subsection{Discussion}
Theorem~\ref{thm:IG}  is a striking contrast to the $1D$-logarithmic interacting case, i.e., {\it infinite-dimensional unlabelled Dyson Brownian motion}, where {\it the absence
of a spectral gap} has been proven in~\cite{Suz24} for the Dirichlet form whose reversible measure is the $\sine_2$ point process. 
This difference is already
visible in the long-wavelength behaviour of their structure factors.
For a stationary point process with structure factor $S(k)$, the
variance and energy of a centred linear statistic $F_f(\gamma)=\langle f, \gamma \rangle - \rho \int f \diff x$ ($\rho$ is the intensity constant and $f$ is smooth and compactly supported) satisfy,
up to the common intensity and Fourier-normalization constants,
$$
  \Var(F_f)
  \simeq
  \int S(k)|\widehat f(k)|^2 \diff k \comma
  \qquad
  \E(F_f)
  \simeq
  \frac12\int |k|^2|\widehat f(k)|^2 \diff k \fstop
$$
Thus a density mode concentrated near frequency $k$ has Rayleigh
quotient~$\E(F_f)/\Var(F_f)$ approximately $|k|^2/(2S(k))$.  For Ginibre,
$$
  S_{\rm Gin}(k)
  =
  1-e^{-|k|^2/4}
  \sim \frac{|k|^2}{4} \comma
$$
hence this quotient converges to $2$ as $k\to0$.  Long-wavelength
density fluctuations therefore do not become arbitrarily slow.  By
contrast, 
$$S_{\sine_2}(k)\sim |k|/(2\pi)\comma$$ so the corresponding
Rayleigh quotient is asymptotic to $\pi|k|$ and tends to zero as $k \to 0$.  This explains the absence of a spectral gap for
$\sine_2$.  
This structure factor calculation however only concerns linear
statistics, which is not enough for proving the existence of spectral gap for the entire domain~$\mathcal D$ in the Ginibre case. In the proof of the current paper, a gap estimate for $\overline{\partial}$-Laplacian (Lemma~\ref{lem:FG}) and a high-dimensional projection estimate (Lemma~\ref{lem:LHV}) for holomorphic functions with Gaussian weight play key roles.

\smallskip
The enhancement of the spectral gaps in the passage from finite-particle to infinite-particle systems in Theorems~\ref{thm:IG} and \ref{prop:FG} has the following mechanism: in the finite particle case, the following function attains the infimum of the Rayleigh quotient:
$$
  G_N(z_1,\ldots,z_N)
  =
  \Re\left(\sum_{i=1}^N z_i\right)\comma
$$
where $\Re$ is the real part of complex numbers. 
This attains the exact spectral gap, which is equal to one. 
In the infinite-particle case, however, there is no $L^2(\mu)$ infinite-particle counterpart of the function $G_N$ belonging to the domain~$\D$. Furthermore, any function defined only on  finite-particle configuration spaces  cannot contribute to the infinite-particle spectral gap  because $\mu(\U^{(N)})=0$ for every $N \ge 0$.
By
contrast, Lemma~\ref{lem:LHV} shows that the
holomorphic component of every fixed bounded local observable
vanishes in the infinite-particle limit.  This is why the infinite-particle spectral gap
improves from one to two.

The mechanism revealed here suggests a more general phenomenon: the lowest finite-particle eigenmodes may be carried by global degrees of freedom, such as the centre of mass, boundary motion, or other collective coordinates, which have no nontrivial counterpart in the infinite-particle case.  Such modes can therefore disappear under the infinite-particle limit, while the remaining spectrum reflects the intrinsic relaxation of bulk fluctuations and may start at a strictly higher value.  
It would be, therefore, interesting to understand this as a general principle for interacting particle systems  beyond the Ginibre model, and to identify structural conditions under which this type of enhancement of spectral gaps occurs. 

\smallskip
For the finite-particle systems, spectral gap inequalities for finite planar Coulomb gases were established by~\cite{bolley2018dynamics}, with constants that may depend on the number of particles. In the determinantal Ginibre case, Theorem~\ref{prop:FG} identifies the optimal constant on the unlabelled sector (i.e., for symmetric functions) and shows its independence of $N$. 

\subsection{Proof outline}
The proofs of Theorems~\ref{thm:IG} and~\ref{prop:FG} have a common mechanism. The key idea is to decompose ground-state transformed functions into the holomorphic part and its orthogonal complement. 
For $F \in \D_N$, 
set~the ground-state transformed function~$v_N:=(F-\int F \diff \mu_N)\Delta_N/\sqrt{Z_N}.$
Let~
$$\mathcal L_N=L^2(\C^N \to \C,e^{-\Phi_N}\diff z^{\otimes N})$$ and $P_N: \mathcal L_N \to \mathcal H_N$ be the orthogonal projection onto~$\mathcal H_N \subset \mathcal L_N$ the subspace of holomorphic functions. Then
\begin{equation}
 \Var_{N}(F) = \|v_N\|_{\mathcal L_N}^2
  =\|v_N-P_Nv_N\|_{\mathcal L_N}^2 +  \|P_Nv_N\|_{\mathcal L_N}^2 \le  \frac12\E_N(F) +  \|P_Nv_N\|_{\mathcal L_N}^2 \fstop
  \label{eqI:AL} 
\end{equation}
The inequality~
\begin{align} 
\|v_N-P_Nv_N\|_{\mathcal L_N}^2 \le  \frac12\E_N(F)
\label{e:OPE}
\end{align}
 follows by the sharp $L^2$-spectral gap estimate of the Gaussian weighted $\overline{\partial}$-Laplacian~$\Box_N$ on the orthogonal complement $\mathcal H_N^{\perp}$ in~$\mathcal L_N$~(Lemma~\ref{lem:FG}).

\smallskip
The proofs of Theorems~\ref{thm:IG} and~\ref{prop:FG} however deviate in the estimate of the holomorphic projection~$\|P_Nv_N\|_{\mathcal L_N}^2$ in~\eqref{eqI:AL}. In Theorem~\ref{prop:FG}, the spectral gap is considered in $\mathcal D_N$, for which the angle estimate (Lemma~\ref{lem:AG}) provides
$$\|P_Nv_N\|_{\mathcal L_N}^2 \le  \|v_N-P_Nv_N\|_{\mathcal L_N}^2.$$
Plugging this into~\eqref{eqI:AL}  together with~\eqref{e:OPE},  
we obtain~
$$\Var_{N}(F)  \le \E_N(F)\fstop$$
The equality is achieved by~$F(z_1,\ldots,z_N)=\Re(\sum_{i=1}^N z_i)$ as mentioned above. 

\smallskip
In~Theorem~\ref{thm:IG}, we instead take a cylinder function $F$. The significant property is that it is a local function, i.e., there is a compact set $B \subset \C$ such that $F(\gamma)=F(\gamma|_{B})$ for every~$\gamma \in \U$. 
We then prove in Lemma~\ref{lem:LHV} that the holomorphic projection decays to zero as the number of particles $N$ tends to infinity:
\begin{equation}\label{eqI:HEB}
  \|P_Nv_N\|_{\mathcal L_N}^2
  \le
  4\|F\|_\infty^2
  \frac{|B|}{\pi}\,
  \varepsilon_N(B) \xrightarrow{N \to \infty}0 \fstop
\end{equation}
Here $\varepsilon_N(B)$ measures the total mass inside~$B$ carried by the holomorphic modes with indices at least~$N$, and it tends to zero for every fixed compact set~$B$ as~$N$ tends to infinity.
By the weak convergence $\mu^{[N]} \to \mu$, we have $\Var^{[N]}(F) \to \Var(F)$ and $\E^{[N]}(F) \to \E(F)$, where $\Var^{[N]}$ and $\E^{[N]}$ are the variance and the Dirichlet form associated with $\mu^{[N]}$. Estimates~\eqref{eqI:AL} and~\eqref{eqI:HEB} then conclude
$$\Var(F) \le \frac12\E(F) \fstop$$
The optimality of this constant is given by the family of linear statistics $F_{f_R}$ with dilated functions~$f_R(z)=f(z/R)$ for $f \in C_c^\infty(\C)$, with which the Rayleigh quotient converges to $2$.

\section{Proof of Theorem~\ref{prop:FG}}

\subsection{Gaussian weighted $L^2$-spaces and holomorphic projections}
Let $\mathcal L_N=L^2(\C^N \to \C,e^{-\Phi_N}\diff z^{\otimes N})$ denote the space of $e^{-\Phi_N}\diff z^{\otimes N}$-equivalence classes of functions~$u: \C^N \to \C$ with 
$$\|u\|_{\mathcal L_N}^2=\int_{\C^N}|u(z)|^2 e^{-\Phi_N(z)}\diff z^{\otimes N}<+\infty \fstop$$ 
Define the inner product
$$\langle u, v\rangle _{\mathcal L_N}:=\int_{\C^N}u(z) \overline{v(z)} e^{-\Phi_N(z)}\diff z^{\otimes N}\fstop$$
Let~$\mathcal H_N \subset \mathcal L_N$ be the subspace of entire holomorphic functions and 
$$\Psi_N:=\frac{1}{\sqrt{Z_N}} \Delta_N \fstop$$
For symmetric $F$, 
set~$u=(F-\int F\diff \mu_N)\Psi_N.$
Then $u$ is antisymmetric and
\begin{equation}
\Var_{N}(F)=\|u\|_{\mathcal L_N}^2.
  \label{eq:VT}
\end{equation}
Since $\Delta_N$ is holomorphic,
$$
  \overline\partial_i u=(\overline\partial_i F)\Psi_N \comma
  \qquad
  \overline\partial_i=\frac12\left(\frac{\partial}{\partial x_i}
       +i\frac{\partial}{\partial y_i}\right) \comma
  \qquad z_i=x_i+iy_i \fstop
$$
For real-valued $F$, we have $|\overline\partial_i F|^2=\frac14|\nabla_i F|^2.$
Consequently
\begin{equation}
  q_N(u):=\sum_{i=1}^N\|\overline\partial_i u\|_{\mathcal L_N}^2
  =\frac{1}{4}\int\sum_{i=1}^N |\nabla_iF|^2 \diff \mu_N  = \frac12 \E_N(F)\fstop
  \label{eq:DET}
\end{equation}

\begin{lemma}
\label{lem:FG}
Let $P_N:\mathcal L_N\to\mathcal H_N$ be the orthogonal projection.  Then, 
\begin{equation}
  \|u-P_Nu\|_{\mathcal L_N}^2
  \le
  \sum_{i=1}^N\|\overline\partial_i u\|_{\mathcal L_N}^2 
  \label{eq:FG}
\end{equation}
for every $u \in D(q_N):=\{u \in \mathcal L_N: \overline\partial_i u \in \mathcal L_N,\ 1 \le i \le N\}$, where $\overline\partial_i$ is understood as the distributional derivative.
\end{lemma}

\begin{proof}
The adjoint of $\overline\partial_i$ in the weighted $L^2$-space~
$\mathcal L_N=L^2(e^{-\Phi_N}\diff z^{\otimes N})$ is 
$$
  \overline\partial_i^{\,*}=-\partial_i+\overline z_i \comma
  \qquad
  \partial_i=\frac12\left(\frac{\partial}{\partial x_i}
       -i\frac{\partial}{\partial y_i}\right) \fstop
$$
The scalar $\overline\partial$-Laplacian is then defined on $C_c^\infty(\C^N \to \C)$ by 
$$
  \Box_N=\sum_{i=1}^N \overline\partial_i^{\,*}\overline\partial_i \fstop
$$
We use the same symbol~$\Box_N$ for the self-adjoint extension in $\mathcal L_N$ associated with the quadratic form~$(q_N, D(q_N))$ in~\eqref{eq:DET}.
 Noting that $\ker{\Box_N}=\mathcal H_N$, 
 the canonical commutation relations
for the Gaussian creation and annihilation operators show that the spectrum of
$\Box_N$ on the orthogonal complement~$\mathcal H_N^\perp$ is $\{1,2,3,\ldots\}$, see, e.g., ~\cite[Theorems 2.3, 2.4]{Has2013}.  Thus, 
$$
  \Box_N\ge I
  \quad\hbox{on} \quad \mathcal H_N^\perp \fstop
$$
Therefore, using the holomorphicity $\overline{\partial}_i P_Nu=0$ for $1 \le i \le N$, 
\begin{equation}
  \|u-P_Nu\|^2_{\mathcal L_N}
  \le \langle \Box_N(u-P_Nu),u-P_Nu\rangle_{\mathcal L_N} = \sum_{i=1}^N \langle \overline{\partial}_i(u-P_Nu),\overline{\partial}_i(u-P_Nu)\rangle_{\mathcal L_N}
  =\sum_{i=1}^N\|\overline\partial_i u\|^2_{\mathcal L_N} \comma
\end{equation}
for every~$u$ in the domain of the operator $\Box_N$. By density,  this extends to $D(q_N)$, which is~\eqref{eq:FG}. 
\end{proof}

\subsection{Angle lemma}
Let $\mathcal H_N^{\anti} \subset \mathcal H_N$ be the subspace of $\mathcal H_N$ consisting of antisymmetric functions.
Every element of $\mathcal H_N^{\anti}$ has the form $h=p\Delta_N$, 
where $p$ is holomorphic and symmetric.  Let
$$
  \Omega_N=\C\Delta_N :=\{c\Delta_N: c \in \C\} \comma
  \qquad
  \mathcal H_{N,0}^{\anti}=\mathcal H_N^{\anti}\ominus \Omega_N \fstop
$$
On the closed subspace $\Delta_N L^2_{\mathrm{sym}}(\mu_N):=\{\Delta_N u: u \in L^2_{\mathrm{sym}}(\mu_N)\}$, define the involution
$$
  J(\Delta_N u)=\Delta_N \overline{u} \fstop
$$
Thus, if $u=F\Delta_N$ with $F$ real-valued, then $Ju=u$.

\begin{lemma}
\label{lem:AG}
Let $F \in L^2(\mu_N)$ be real-valued and symmetric, and set
$u=(F-\int F \diff \mu_N)\Psi_N$.  Then
\begin{equation}
  \|u\|_{\mathcal L_N}^2
  \le
  2\|u-P_Nu\|^2_{\mathcal L_N} \fstop
  \label{eq:AL}
\end{equation}
\end{lemma}

\begin{proof}
Define $h=P_Nu$ and $r=u-h$. Note that the permutation operator $U_\sigma: \mathcal L_N \to \mathcal L_N$ defined as $(U_\sigma u)(z_1, \ldots, z_N):=u(z_{\sigma(1)},\ldots, z_{\sigma(N)})$ is the unitary operator for every permutation $\sigma$ since the Gaussian weight~$e^{-\Phi_N}$ is permutation-invariant. Thus, $P_N \circ U_\sigma=U_\sigma\circ P_N$.  Together with that $u$ is antisymmetric and $U_\sigma$ preserves the holomorphicity, we thus obtain
$h\in\mathcal H_N^{\anti}$.  
Thus
$$
  0=\int \Bigl(F-\int F \diff \mu_N\Bigr)\diff \mu_N
  =\int \Bigl(F-\int F \diff \mu_N\Bigr) |\Psi_N|^2e^{-\Phi_N}\diff z^{\otimes N}
  =\frac{1}{\sqrt{Z_N}}\langle u,\Delta_N\rangle_{\mathcal L_N} \fstop
$$
Noticing $\langle u,\Delta_N\rangle_{\mathcal L_N} = \langle r+h,\Delta_N\rangle_{\mathcal L_N} = \langle h,\Delta_N\rangle_{\mathcal L_N}$, we have that $h=P_Nu$ is also orthogonal to $\Delta_N$, hence
$h\in\mathcal H_{N,0}^{\anti}$.

\smallskip
We next show the orthogonality
\begin{equation}
  J\mathcal H_{N,0}^{\anti}\perp \mathcal H_N^{\anti} \fstop
  \label{eq:JO}
\end{equation}
It is enough to check this on holomorphic polynomials and then use density.
Take
$$
  h_1=p\Delta_N\in\mathcal H_{N,0}^{\anti}\comma
  \qquad
  h_2=q\Delta_N\in\mathcal H_N^{\anti}\comma
$$
with $p,q$ symmetric holomorphic polynomials.  Since $\mathcal H_{N,0}^{\anti} \perp \Omega_N$ by definition, 
$h_1\perp\Delta_N$, which implies that the constant term of the polynomial~$p$ is zero:
 indeed, decompose~$p=\sum_{m=0}^\infty p_m$, where $p_m$ is homogeneous part of degree $m$. Then, $p_m(e^{i\theta}z_1, \ldots, e^{i\theta}z_N)=e^{im\theta}p_m(z_1, \ldots, z_N)$. Noting that the measure $|\Delta_N|^2e^{-\Phi_N}\diff z^{\otimes N}$ is invariant under the push-forward by the rotation~$(z_1, \ldots, z_N) \mapsto (e^{i\theta}z_1, \ldots, e^{i\theta}z_N)$, the integral $I_m=\int_{\C^N} p_m|\Delta_N|^2e^{-\Phi_N}\diff z^{\otimes N}$ satisfies the following property for every~$\theta$:
$$
 I_m=e^{im\theta}I_m \fstop
$$
Thus, we have $I_m=0$ for $m \ge 1$, and only the term $I_0$ can contribute. By the orthogonality~$h_1\perp\Delta_N$, 
$$0=\langle h_1, \Delta_N \rangle_{\mathcal L_N} = \langle p \Delta_N, \Delta_N \rangle_{\mathcal L_N}=I_0=Z_Np_0 \fstop$$

With our convention of the inner-product, 
$$
  \langle Jh_1,h_2\rangle_{\mathcal L_N}
  =\int_{\C^N}\overline{p(z)q(z)}\,|\Delta_N(z)|^2e^{-\Phi_N(z)}\diff z^{\otimes N} \fstop
$$
Every homogeneous component
of $p q$ has strictly positive total degree, because $p$ has no constant term.
Using the same argument as proving $I_m=0$ for $m \ge 1$ therefore gives $\langle Jh_1,h_2\rangle_{\mathcal L_N}=0$.  This
proves \eqref{eq:JO}.

\smallskip
We now use $Ju=u$ thanks to that $F$ is real-valued.  Since $u=h+r$,
$$
  h+r=Ju=Jh+Jr \fstop
$$
As we saw in the first paragraph, using $P_N \circ U_\sigma=U_\sigma\circ P_N$, 
the projection $P_N$ preserves antisymmetry.  Thus, recalling that $Jh$ is antisymmetric, $P_NJh \in \mathcal H_N^{\anti}$.
Due to~\eqref{eq:JO}, $Jh \perp  \mathcal H_N^{\anti}$, thus, for every $g \in \mathcal H_N^{\anti}$, the orthogonality of $(I-P_N)$ to $\mathcal H_N$ implies
$$\langle P_N J h, g \rangle_{\mathcal L_N}=\langle P_N J h, g \rangle_{\mathcal L_N}+\langle(I-P_N) J h, g \rangle_{\mathcal L_N} = \langle J h, g \rangle_{\mathcal L_N} =0 \fstop$$
Thus, $P_NJh \perp \mathcal H_N^{\anti}$. Since $P_NJh \in \mathcal H_N^{\anti}$, we obtain
$$
  P_NJh=0 \fstop
$$
Projecting $h+r=Jh+Jr$ onto $\mathcal H_N$ gives
$$
  h=P_NJr \fstop
$$
Hence, using that $P_N$ is a contraction and $J$ is  norm-preserving,
$$
  \|h\|_{\mathcal L_N} \le \|Jr\|_{\mathcal L_N} =\|r\|_{\mathcal L_N}  \fstop
$$
Since $r=u-P_Nu$ is orthogonal to $h=P_Nu$,
$$
  \|u\|^2_{\mathcal L_N} =\|h\|^2_{\mathcal L_N} +\|r\|^2_{\mathcal L_N} \le 2\|r\|^2_{\mathcal L_N}  \fstop
$$
This is \eqref{eq:AL}.
\end{proof}

\subsection{Proof of Theorem~\ref{prop:FG}}
Combining Lemmas \ref{lem:FG} and \ref{lem:AG},
$$
  \|u\|_{\mathcal L_N}^2
  \le
  2\sum_{i=1}^N\|\overline\partial_i u\|_{\mathcal L_N}^2 \fstop
$$
Using \eqref{eq:VT} and \eqref{eq:DET},
$$
  \Var_{N}(F)
  \le
  2\cdot\frac{1}{4}
  \int\sum_{i=1}^N |\nabla_iF|^2 \diff \mu_N 
  =\mathcal E_N(F) \fstop
$$
The extension to the closure~$\mathcal D_N$ is standard by density.

\paragraph{\bf Optimality} 
Set
$G_N(z_1,\ldots,z_N)
  =
  \Re\left(\sum_{i=1}^N z_i\right).$
Writing~$\mssm=\frac1N\sum_{i=1}^N z_i$ with $y_i=z_i-\mssm$,
we have
$$  \sum_{i=1}^N y_i=0 \comma
  \qquad
  \Delta_N(z_1,\ldots,z_N)
  =
  \Delta_N(y_1,\ldots,y_N) \comma
$$
and
$$
  \sum_{i=1}^N|z_i|^2
  =
  \sum_{i=1}^N|y_i|^2+N|\mssm|^2 \fstop
$$
Thus the centre-of-mass variable $\mssm \in \C$ is independent of the relative
coordinates and has density
$$
  \frac{N}{\pi}e^{-N|\mssm|^2} \diff \mssm \comma
$$
namely,
$$
  \int_{\C^N} f(\mssm(z)) \diff \mu_N(z) =  \int_{\C} f(z)  \frac{N}{\pi}e^{-N|z|^2} \diff z \comma
$$
 for every $f: \C \to \R$ whenever the RHS makes sense. 
Hence
$$
  \Var_{N}(G_N)
  =
  N^2\Var_{N}(\Re \mssm) = N^2 \int_{\R}x^2\sqrt{\frac N\pi}e^{-Nx^2} \diff x
  =
  \frac N2 \fstop
$$
On the other hand, by the standard cut-off argument, $G_N \in \mathcal D_N$, and 
$$
  |\nabla_iG_N|^2=1
  \qquad (1\le i\le N)\comma
$$
so that
$$
  \mathcal E_N(G_N)
  =
  \frac12\int_{\C^N}
  \sum_{i=1}^N|\nabla_iG_N|^2 \diff\mu_N
  =
  \frac N2 \fstop
$$
Therefore
$$
  \frac{\mathcal E_N(G_N)}
       {\Var_{N}(G_N)}
  =1 \comma
$$
and the finite-particle spectral gap is exactly equal to $1$ for every $N \ge 1$. 
The proof of  Theorem~\ref{prop:FG} is complete. 

\begin{remark}[Symmetry does matter] \label{rem:SDM}
Theorem~\ref{prop:FG} is a statement only on symmetric functions.  The proof of Lemma~\ref{lem:AG} essentially relies on the symmetry of $F$. 
A uniform spectral gap inequality for non-symmetric functions on $\C^N$  is indeed false.    For example, the non-symmetric observable
$F(z_1,\ldots,z_N)=\Re z_1$ has variance of order $N$, while its
Dirichlet energy is of order one.  Thus the  spectral gap is at most of order~$1/N$, which is not uniformly bounded below.
\end{remark}

\section{Proof of Theorem~\ref{thm:IG}} \label{s:PT}
\subsection{Projection estimate}
For simplicity, we use the notation~
$$m(\diff z)=e^{-|z|^2}\diff z \fstop$$ 
The functions
$$
  \varphi_k(z)=\frac{z^k}{\sqrt{\pi k!}},
  \qquad k=0,1,2,\ldots \comma
$$
form an orthonormal basis of 
$$\mathcal H_1= \left\{ f\in L^2(\C \to \C, m): f\text{ is entire holomorphic} \right\} \fstop$$  
Recalling~$\Psi_N=\frac{\Delta_N}{\sqrt{Z_N}}$, we have 
$$
  |\Psi_N|^2m^{\otimes N}=\mu_N \fstop
$$
We have the determinant expression
$$
  \Psi_N(z_1, \ldots, z_{N})
  =
  (-1)^{\frac{N(N-1)}{2}}\varphi_0\wedge\varphi_1\wedge\cdots
  \wedge\varphi_{N-1} (z_1, \ldots, z_{N}) \comma
$$
where $S_N$ is the $N$-permutation group and 
$$
\varphi_0\wedge\varphi_1\wedge\cdots
  \wedge\varphi_{N-1} (z_1, \ldots, z_{N}) 
  := \frac{1}{\sqrt{N!}}\det{[\varphi_{i-1}(z_{j})]}_{i, j = 1}^{N}
  = \frac{1}{\sqrt{N!}}\sum_{\sigma \in S_N}\sgn(\sigma)\prod_{i=1}^N \varphi_{i-1}(z_{\sigma(i)})
   \fstop
$$
For a compact set $B\subset\C$, define
\begin{equation}\label{eq:TM}
  \varepsilon_N(B)
  =
  \sum_{k=N}^{\infty}
  \int_B|\varphi_k(z)|^2\,m(\diff z)
  =
  \frac1\pi
  \int_B e^{-|z|^2}
  \sum_{k=N}^{\infty}
  \frac{|z|^{2k}}{k!} \diff z \fstop
\end{equation}
If $B\subset\{|z|\le R\}$, then
\begin{equation}\label{eq:TMV}
  \varepsilon_N(B)
  \le
  \frac{|B|}{\pi}
  \sum_{k=N}^{\infty}\frac{R^{2k}}{k!}
  \longrightarrow0 \comma
\end{equation}
where
$|B|=\int_B\diff z.$

For a compact set~$B \subset \C$, we say that a function~$F: \U \to \R$ is {\it $B$-local}  if $F(\gamma)=F(\gamma|_{B})$ for every $\gamma \in \U$.
Recall that $\mu^{[N]}$ is the push-forward $(\pi_N)_\# \mu_N$ on~$\U^{(N)}$. For every~$\mu^{[N]}$-integrable $F:  \U^{(N)} \to \C$
$$\int_{\C^N} \tilde F( z_1,\ldots, z_N) \diff \mu_N(z_1, \ldots, z_N)= \int_{\U^{(N)}} F(\gamma) \diff \mu^{[N]}(\gamma) \comma $$
 where $\tilde F$ is  {\it the symmetric representative of $F$} defined as~$\tilde F(z_1, \ldots, z_N):=F(\sum_{i=1}^N \delta_{z_i})$.
\begin{lemma}
\label{lem:LHV}
Let $F:\U \to \R$ be a bounded real-valued Borel-measurable $B$-local function for some compact
set $B\subset\C$.   Let $F_N: \C^N \to \R$ be the symmetric representative of $F|_{\U^{(N)}}$ for $N \ge 2$. 
Let
$$
  v_N
  =
  \left(
    F_N-\int  F_N\diff \mu_N
  \right)\Psi_N
  \cquad
  h_N=P_Nv_N \fstop
$$
Then
\begin{equation}\label{eq:HEB}
  \|h_N\|_{\mathcal L_N}^2
  \le
  4\|F\|_\infty^2
  \frac{|B|}{\pi}\,
  \varepsilon_N(B) \fstop
\end{equation}
In particular,
$$
  \|h_N\|_{\mathcal L_N}\longrightarrow0
  \qquad\text{as }N\to\infty \fstop
$$
\end{lemma}

\begin{proof}
\paragraph{Step 1. annihilation and determinant}
Let $\mathcal L_N^{\anti} \subset \mathcal L_N$ be the subspace of antisymmetric elements. 
For $k\geq0$, define the fermionic annihilation operator
$$
  a_k:\mathcal L_N^{\anti}\longrightarrow
  \mathcal L_{N-1}^{\anti}
$$
defined as
$$
  (a_kg)(z_1,\ldots,z_{N-1})
  :=
  \sqrt N
  \int_{\C}
  \overline{\varphi_k(z)}
  g(z,z_1,\ldots,z_{N-1})\,m(\diff z) \fstop
$$

Let $\mathcal I_N=\{I \subset \N_0: \# I =N\}$. 
For~$I=\{i_1, \cdots, i_N\} \in \mathcal I_N$ with~$i_1 < \ldots < i_N \in \N_0$, 
let $\varphi_I$ be the determinant
$$
  \varphi_I
  =
   \varphi_{i_1}\wedge\cdots\wedge\varphi_{i_N} \fstop
$$
By the orthogonality of $(\varphi_{k})_{k \in \N_0}$, 
\begin{equation} \label{e:DAP}
\|a_k\varphi_I\|^2_{\mathcal L_{N-1}} = \1_{\{k \in I\}}  \fstop
\end{equation}
The functions $\{\varphi_I\}_{I \in \mathcal I_N}$ form an orthonormal basis of~$\mathcal H_N^{\anti}$.  If
$$
  h=\sum_{I \in \mathcal I_N} c_I\varphi_I \comma
$$
then, by~\eqref{e:DAP}, 
\begin{equation}\label{eq:NTI}
  \sum_{k=N}^{\infty}\|a_kh\|^2_{\mathcal L_{N-1}}
  =
  \sum_{I \in \mathcal I_N}
  \#\bigl(I\cap\{N,N+1,\ldots\}\bigr)|c_I|^2 \fstop
\end{equation}
The only $N$-element subset of $\{0,1,\ldots,N-1\}$ is
$$
  I_0=\{0,1,\ldots,N-1\} \comma
$$
and $\varphi_{I_0}=  (-1)^{\frac{N(N-1)}{2}} \Psi_N$.  Hence, if
$h\perp\Psi_N$ then $c_{I_0}=0$ and  every nonzero term in its expansion contains at least
one index greater than or equal to $N$, i.e., $\#\bigl(I\cap\{N,N+1,\ldots\}\bigr) \ge 1$ whenever $I \neq I_0$.  Thus, it follows from
\eqref{eq:NTI} that if
$h\perp\Psi_N$, then
\begin{equation}\label{eq:NTL}
  \sum_{k=N}^{\infty}\|a_kh\|^2_{\mathcal L_{N-1}}
  =
  \sum_{I \in \mathcal I_N}
  \#\bigl(I\cap\{N,N+1,\ldots\}\bigr)|c_I|^2 \ge   \sum_{I \in \mathcal I_N \setminus \{I_0\}} |c_I|^2 = \|h\|^2_{\mathcal L_N} \fstop
\end{equation}

\smallskip
\noindent\paragraph{Step 2. estimate of $h_N$} Since $\Psi_N$ is holomorphic,
$$
  \langle h_N,\Psi_N\rangle_{\mathcal L_N}
  =
  \langle v_N,\Psi_N\rangle_{\mathcal L_N}
  =
  \int
  \left(
    F_N-\int_{\C^N} F_N \diff \mu_N
  \right)\diff \mu_N
  =
  0 \fstop
$$
Thus $h_N\perp\Psi_N$ and we may therefore apply \eqref{eq:NTL} to $h_N$.
The annihilation operator by a holomorphic mode commutes with the
holomorphic projection in the following sense:
$$
  a_kP_N=P_{N-1}a_k \fstop
$$
Indeed, for $f \in \mathcal L_{N}^{\anti} \quad g \in \mathcal H_{N-1}^{\anti}$ 
$$\langle a_kP_Nf, g \rangle_{\mathcal L_{N-1}} =\langle P_Nf, a_k^*g \rangle_{\mathcal L_N} = \langle f, a_k^*g \rangle_{\mathcal L_N}  = \langle a_kf, g \rangle_{\mathcal L_{N-1}} =\langle P_{N-1}a_kf, g \rangle_{\mathcal L_{N-1}}  \comma$$
 where the second equality follows from $a_k^*g= \varphi_k \wedge g \in \mathcal H_{N}^{\anti}$, which implies $\langle P_Nf, a_k^*g \rangle_{\mathcal L_N} = \langle f-(I-P_N)f, a_k^*g \rangle_{\mathcal L_N}   = \langle f, a_k^*g \rangle_{\mathcal L_N}$, and the last equality follows similarly. 
 Thus, $(a_kP_Nf-P_{N-1}a_kf) \perp \mathcal H_{N-1}^{\anti}$.  As $P_{N-1}$ preserves antisymmetry, we have~$a_kP_Nf-P_{N-1}a_kf \in  \mathcal H_{N-1}^{\anti}$, which concludes $a_kP_Nf = P_{N-1}a_kf$.
 
Consequently, together with \eqref{eq:NTL} and the contraction property of the projection $P_{N-1}$, 
\begin{equation}\label{eq:PAB}
  \|h_N\|^2_{\mathcal L_N}
  \le
  \sum_{k=N}^{\infty}\|a_kh_N\|^2_{\mathcal L_{N-1}}
  \le
  \sum_{k=N}^{\infty}\|a_kv_N\|^2_{\mathcal L_{N-1}} \fstop
\end{equation}

\smallskip
\noindent\paragraph{Step 3. annihilation estimate I} For fixed
$$
  \eta=(z_1,\ldots,z_{N-1}) 
  \cquad
  \gamma_\eta=\sum_{i=1}^{N-1}\delta_{z_i} \comma
$$
we have
$$
  v_N(z,\eta)
  =
  \left(
    F(\gamma_\eta+\delta_z)-\int_{\C^N} F_N\diff \mu_N
  \right)\Psi_N(z,\eta) \fstop
$$
Hence, by the definition of $a_k$,
\begin{equation} \label{e:AEN}
  (a_kv_N)(\eta)
  =
  \sqrt N
  \int_{\C}
  \overline{\varphi_k(z)}
  \left(
    F(\gamma_\eta+\delta_z)-\int_{\C^N} F_N \diff \mu_N
  \right)
  \Psi_N(z,\eta)\,m(\diff z) \fstop
\end{equation}
For $k\geq N$, the mode $\varphi_k$ is not occupied by $\Psi_N$, 
therefore
$$
  (a_k\Psi_N)(\eta)
  =
  \sqrt N
  \int_{\C}
  \overline{\varphi_k(z)}
  \Psi_N(z,\eta)\,m(\diff z)
  =
  0 \fstop
$$
Since
  $F(\gamma_\eta)-\int_{\C^N} F_N \diff \mu_N$
is independent of $z$, multiplying the last identity by this quantity
shows that
$$
  \sqrt N
  \int_{\C}
  \overline{\varphi_k(z)}
  \left(
    F(\gamma_\eta)-\int_{\C^N} F_N \diff \mu_N
  \right)
  \Psi_N(z,\eta)\,m(\diff z)
  =
  0 \fstop
$$
We may subtract this zero term from the preceding expression~\eqref{e:AEN} for
$a_kv_N$.  The mean terms then cancel, and we obtain
\begin{equation}\label{eq:AID}
  (a_kv_N)(\eta)
  =
  \sqrt N
  \int_{\C}
  \overline{\varphi_k(z)}
  D_zF(\gamma_\eta)
  \Psi_N(z,\eta)\,m(\diff z) \cquad k \ge N \cquad
\end{equation}
where
$
  D_zF(\gamma_\eta)
  :=
  F(\gamma_\eta+\delta_z)-F(\gamma_\eta).
$
Since $F$ is a $B$-local function,
$$
  D_zF(\gamma_\eta)=0
  \qquad\text{for }z\notin B \fstop
$$
For fixed $\eta$, define
$$
  G_\eta(z)
  =
  \mathbf 1_B(z)\,
  D_zF(\gamma_\eta)\,
  \Psi_N(z,\eta) \fstop
$$
Since $D_zF(\gamma_\eta)=0$ for $z\notin B$, the identity
\eqref{eq:AID} can be written, for every
$k\geq N$, as
$$
  (a_kv_N)(\eta)
  =
  \sqrt N
  \int_{\C}
  G_\eta(z)\overline{\varphi_k(z)}\,m(\diff z)
  =
  \sqrt N\,
  \langle G_\eta,\varphi_k\rangle_{L^2(m)} \fstop
$$
Consequently,
$$
  |(a_kv_N)(\eta)|^2
  =
  N
  \left|
    \langle G_\eta,\varphi_k\rangle_{L^2(m)}
  \right|^2 \fstop
$$

\smallskip
\noindent\paragraph{Step 4. annihilation estimate II} Let $Q_N$ denote the orthogonal projection in $L^2(m)$ onto
$ \overline{{\rm span}}
  \{\varphi_N,\varphi_{N+1},\ldots\}$.
Thus
$$
  Q_NG_\eta
  =
  \sum_{k=N}^{\infty}
  \langle G_\eta,\varphi_k\rangle_{L^2(m)}
  \varphi_k \fstop
$$
Since the functions $\{\varphi_k\}_{k\geq N}$ are orthonormal,
Parseval's identity gives
$$
  \|Q_NG_\eta\|_{L^2(m)}^2
  =
  \sum_{k=N}^{\infty}
  \left|
    \langle G_\eta,\varphi_k\rangle_{L^2(m)}
  \right|^2 \fstop
$$
Therefore,
\begin{equation}\label{eq:PT}
  \sum_{k=N}^{\infty}|(a_kv_N)(\eta)|^2
  =
  N\|Q_NG_\eta\|_{L^2(m)}^2 \fstop
\end{equation}
Since $G_\eta$ is supported in $B$, we have
$
  G_\eta=\mathbf 1_BG_\eta$, 
where $\mathbf 1_B$ is regarded as the multiplication operator on
$L^2(m)$. Hence
$$
  Q_NG_\eta
  =
  Q_N\mathbf 1_BG_\eta \fstop
$$
By the definition of the operator norm,
$$
  \|Q_N\mathbf 1_BG_\eta\|_{L^2(m)}
  \le
  \|Q_N\mathbf 1_B\|_{\mathrm{op}}
  \|G_\eta\|_{L^2(m)} \fstop
$$
It follows that
\begin{equation}\label{eq:TOB}
  \|Q_NG_\eta\|_{L^2(m)}^2
  \le
  \|Q_N\mathbf 1_B\|_{\mathrm{op}}^2
  \|G_\eta\|_{L^2(m)}^2 \fstop
\end{equation}
Moreover, the Hilbert--Schmidt norm gives
$$
  \|Q_N\mathbf 1_B\|_{\mathrm{op}}^2
  \le
  \|Q_N\mathbf 1_B\|_{\mathrm{HS}}^2
  =
  \sum_{k=N}^{\infty}
  \int_B|\varphi_k(z)|^2\,m(\diff z)
  =
  \varepsilon_N(B) \fstop
$$
Combining this with \eqref{eq:PT} and
\eqref{eq:TOB}, we obtain
$$
  \sum_{k=N}^{\infty}|(a_kv_N)(\eta)|^2
  \le
  N\varepsilon_N(B)\|G_\eta\|_{L^2(m)}^2 \fstop
$$
Therefore, by the definition of $G_\eta$,
$$
  \sum_{k=N}^{\infty}|(a_kv_N)(\eta)|^2
  \le
  N\varepsilon_N(B)
  \int_B
  |D_zF(\gamma_\eta)|^2
  |\Psi_N(z,\eta)|^2\,m(\diff z) \fstop
$$

\smallskip
\noindent\paragraph{Step 5. conclusion} Integrating over $\eta$ and using
$|\Psi_N|^2m^{\otimes N}=\mu_N$, we obtain
\begin{align}
  \sum_{k=N}^{\infty}\|a_kv_N\|^2_{\mathcal L_{N-1}} 
  & \le N  \varepsilon_N(B)
  \int \1_{B}(z)
  |D_zF(\gamma_\eta)|^2
  |\Psi_N(z,\eta)|^2\,m(\diff z) \, m^{\otimes N-1}(\diff \eta) =:(I)
  \end{align}
  Taking $\gamma:=  \gamma_\eta + \delta_z$, we have $D_zF(\gamma_\eta)=F(\gamma_\eta +\delta_z) - F(\gamma_\eta) = F(\gamma)-F(\gamma-\delta_z)$.
 Writing $\gamma_z=\sum_{i=1}^N\delta_{z_i}$,  the RHS above can be further deduced to
  \begin{align}
  (I) &= N \varepsilon_N(B)
  \int \1_{B}(z_1)
  |F(\gamma_z)-F(\gamma_z-\delta_{z_1})|^2
  \,\mu_N(\diff z) 
  \\
   &= \varepsilon_N(B)
  \int \sum_{i=1}^N \1_{B}(z_i)
  |F(\gamma_z)-F(\gamma_z-\delta_{z_i})|^2
  \,\mu_N(\diff z)
  \\
  & = \varepsilon_N(B)
  \int
  \sum_{x\in\gamma\cap B}
  |F(\gamma)-F(\gamma-\delta_x)|^2
  \,\mu^{[N]}(\diff \gamma)
  \notag\\
  &\le
  4\|F\|_\infty^2
  \varepsilon_N(B)
  \int\gamma(B)\,\mu^{[N]}(\diff \gamma) \fstop
  \label{eq:AFB}
\end{align}
Since the correlation kernel density $K_N$ of $\mu^{[N]}$ is bounded by
\begin{equation} \label{e:KEN}
  K_N(z,z)
  =
  \frac1\pi e^{-|z|^2}
  \sum_{k=0}^{N-1}\frac{|z|^{2k}}{k!} \le \frac1\pi e^{-|z|^2}
  \sum_{k=0}^{\infty}\frac{|z|^{2k}}{k!} = \frac1\pi e^{-|z|^2} e^{|z|^2} 
  \le\frac1\pi \comma
\end{equation}
the intensity formula~$\int \gamma(B) \, \mu^{[N]}(\diff \gamma) = \int_B K_N(z,z) \diff z$  gives
\[
  \int\gamma(B)\,\mu^{[N]}(\diff \gamma)
  =
  \int_BK_N(z,z) \diff z
  \le\frac{|B|}{\pi} \fstop
\]
Combining this estimate with~\eqref{eq:PAB} and~\eqref{eq:AFB},  we conclude~\eqref{eq:HEB}.
The convergence follows from~\eqref{eq:TMV}.
\end{proof}

\subsection{Proof of Theorem~\ref{thm:IG}}
\begin{proof}[Proof of Theorem~\ref{thm:IG}]
\paragraph{Step 1. Local convergence} Recall that $\mu^{[N]}$ is the probability measure on $\U^{(N)}$ that  is the law of the $N$-point determinantal point process in $\C$ whose correlation kernel density is given by
$$
  K_N(z,w)
  =\frac1\pi
  \exp\left\{-\frac{|z|^2}{2}-\frac{|w|^2}{2}\right\}
  \sum_{k=0}^{N-1}\frac{(z\overline w)^k}{k!} \fstop
$$
For every compact set $B\subset\C$,
$$
  K_N\longrightarrow K
  \qquad\text{uniformly on }B\times B \fstop
$$
Indeed, this follows directly from the locally uniform convergence of the
Taylor series for the exponential.  Consequently, for every $m\in\N$,
the $m$-point correlation functions of $\mu^{[N]}$
$$
  \rho_N^{(m)}(z_1,\ldots,z_m)
  =
  \det\bigl[K_N(z_i,z_j)\bigr]_{i,j=1}^m
$$
converge uniformly on $B^m$ to the corresponding correlation
function of $\mu$.

The law of the restricted point process~$\gamma_N|_{B}$, where $\gamma_N \sim \mu^{[N]}$,  converges weakly to the law of $\gamma|_{B}$ where $\gamma \sim \mu$. 
 This follows, for example,
from the general convergence criterion for determinantal point fields
in \cite[Theorem~5]{SoshnikovDPP}.  
Indeed, the operators associated
with $K_N$ converge strongly to the operator associated with
$K$, while for every bounded set $A \subset \C$, 
$$
  \Tr(\mathbf 1_AK_N\mathbf 1_A)
  =
  \int_A K_N(z,z) \diff z
  \longrightarrow
  \int_A K(z,z) \diff z \fstop
$$
Moreover, the local particle numbers have uniformly bounded
exponential moments.  To see this, let $\gamma_N(B)$ denote the number of
particles in any fixed  compact set~$B$.  By the Bernoulli representation for determinantal
point processes
\cite[Theorem~7]{HoughKrishnapurPeresVirag},
$\gamma_N(B)$ has the same distribution as a sum of independent Bernoulli
random variables whose parameters are the eigenvalues $(\lambda_i^{(N)})_{i \in \N}$ of
the operator~$\mathbf 1_BK_N\mathbf 1_B$.  Recalling $K_N(z,z)\le \frac1\pi$ in~\eqref{e:KEN} and the inequality~$1+s \le e^s$ for $s \ge0$, we have that, 
for every $t\geq0$,
\begin{align*}
  \sup_N \mathbb E\bigl[e^{t\gamma_N(B)}\bigr] 
  &=   \sup_N \prod_{i=1}^\infty \Bigl( 1 + (e^t-1)\lambda_i^{(N)}\Bigr) \le   \sup_N \prod_{i=1}^\infty \exp\Bigl(\bigl\{ e^t-1\bigr\}\lambda_i^{(N)}\Bigr)
 \\
 & \le
  \exp\left\{
    (e^t-1)\sup_N\Tr(\mathbf 1_BK_N\mathbf 1_B)
  \right\}
  \le
  \exp\left\{
    \frac{e^t-1}{\pi}|B|
  \right\} \fstop
\end{align*}
In particular, all moments of $\gamma_N(B)$ are bounded uniformly in $N$.
It follows that expectations converge for continuous local observables
which are bounded by a polynomial in the number of particles in a
fixed compact set.

\noindent\paragraph{Step 2. Convergence on cylinder functions}
Take a cylinder function
$$
  F(\gamma)=\Theta\bigl(\langle f_1,\gamma\rangle,\ldots,
                   \langle f_m,\gamma\rangle\bigr) \fstop
$$
Let $\Var^{[N]}(F):=\int F^2 \diff \mu^{[N]} - (\int F \diff \mu^{[N]})^2$.
Since $F$ is a bounded, continuous and local function, 
the weak convergence of $\mu^{[N]}$ to $\mu$ implies
$$
  \Var^{[N]}(F)\longrightarrow \Var(F) \fstop
$$

We now discuss the energy part. 
Set
$$
  \Gamma(F)(\gamma)
  =\frac12\sum_{x\in\gamma}|\nabla_xF(\gamma)|^2 \fstop
$$
The function $\Gamma(F)$ is continuous with respect to the vague topology~$\tau_{v}$ in $\U$ (i.e.,~the topology induced by duality of compactly supported continuous real functions in~$\C$).
Since the $f_j$ have compact support and the first derivatives of $\Theta$ are
bounded, $\Gamma(F)$ is bounded by a constant times the number of particles in a
fixed compact set.  The kernels $K_N$ are locally uniformly bounded, so the
particle numbers in this compact set have uniformly bounded moments, which gives the uniform
integrability of  $\Gamma(F)$ regarding $\mu^{[N]}$.
Thus, combined with the weak convergence of $\mu^{[N]}$ to $\mu$ as probability measures in~$(\U, \tau_v)$,  
$$
  \mathcal E^{[N]}(F):=\int\Gamma(F)\diff\mu^{[N]}
  \longrightarrow
  \int\Gamma(F)\diff\mu
  =\mathcal E(F) \fstop
$$

\noindent\paragraph{Step~3. Spectral gap for infinite particles} 
We first prove the lower bound on the spectral gap.  Let $F$ be a  cylinder function. In particular, there is a compact set $B \subset \C$ such that $F$ is $B$-local. 
For each $N$, take the symmetric representative $F_N$ of $F|_{\U^{(N)}}$ and set
$$
  v_N
  =
  \left(
    F_N-\int F_N\diff \mu_N
  \right)\Psi_N
$$
and write
$$
  h_N=P_Nv_N
  \cquad
  r_N=v_N-h_N \fstop
$$
Since $h_N$ is the holomorphic projection of $v_N$,
we have $
  h_N\perp r_N
$
and
$$
  \Var^{[N]}(F)
  =
  \|v_N\|_{\mathcal L_N}^2
  =
  \|h_N\|_{\mathcal L_N}^2
  +
  \|r_N\|_{\mathcal L_N}^2 \fstop
$$
Lemma~\ref{lem:FG} gives
$$
  \|r_N\|_{\mathcal L_N}^2
  \le
  \sum_{i=1}^N
  \|\overline\partial_i v_N\|_{\mathcal L_N}^2 \fstop
$$
Moreover, $\Psi_N$ is holomorphic and $F$ is real-valued, so
$$
  \sum_{i=1}^N
  \|\overline\partial_i v_N\|_{\mathcal L_N}^2
  =
  \frac14
  \int_{\C^N}
  \sum_{i=1}^N|\nabla_iF_N|^2 \diff \mu_N
  =
    \frac12
  \int_{\U^{(N)}} \Gamma(F) \diff \mu^{[N]}
  =
  \frac12\mathcal E^{[N]}(F) \fstop
$$
Consequently,
\begin{equation}\label{eq:FSE}
  \Var^{[N]}(F)
  \le
  \frac12\mathcal E^{[N]}(F)
  +
  \|h_N\|_{\mathcal L_N}^2 \fstop
\end{equation}
Lemma~\ref{lem:LHV} yields
$$
  \|h_N\|_{\mathcal L_N}^2
  \le
  4\|F\|_\infty^2
  \frac{|B|}{\pi}\varepsilon_N(B) \comma
$$
and the right-hand side tends to zero.
By Step 2, 
$$
  \Var^{[N]}(F)
  \longrightarrow
  \Var(F) \cquad  \mathcal E^{[N]}(F)
  \longrightarrow
  \mathcal E(F) \fstop
$$
Letting $N\to\infty$ in
\eqref{eq:FSE}, we have
\begin{equation} \label{eq:SPG}
  \Var(F)
  \le
  \frac12\mathcal E(F)
\end{equation}
for every cylinder function~$F$. The extension to the closure $\mathcal D$ is standard by density argument.

\smallskip
\noindent\paragraph{Step~4. Optimality}  It remains to prove sharpness of the spectral gap $2$.  Fix a non-constant real-valued function
$f\in C_c^\infty(\C \to \R)$ and define
$
  f_R(z)=f(z/R)
$
and consider the centred linear statistic
$$
  F_R(\gamma)
  =
  \langle f_R,\gamma\rangle
  -
  \frac1\pi\int_{\C}f_R(z) \diff z \fstop
$$
By a standard cut-off argument, $F_R \in \mathcal D$. 
The variance formula for the Ginibre point process gives
\begin{equation}\label{eq:SL}
  \Var(F_R)
  =
  \frac1{2\pi^2}
  \iint_{\C^2}
  |f_R(z)-f_R(w)|^2
  e^{-|z-w|^2} \diff z \diff w \fstop
\end{equation}
On the other hand, by the intensity formula 
$$\int_\U \langle g, \gamma \rangle \mu(\diff \gamma) = \int_{\C} g K(z,z)\diff z =   \frac1\pi \int_{\C} g \diff z \qquad g \in C_c(\C \to \R) \comma
$$
we have
\begin{equation}\label{eq:SLE}
  \mathcal E(F_R)
  =
  \frac1{2\pi}
  \int_{\C}|\nabla f_R|^2 \diff z
  =
  \frac1{2\pi}
  \int_{\C}|\nabla f|^2 \diff z \fstop
\end{equation}
In \eqref{eq:SL}, let
$$
  z=Rx
  \cquad
  w=z+y \fstop
$$
Then
$$
  \Var(F_R)
  =
  \frac{R^2}{2\pi^2}
  \iint_{\C^2}
  \left|
    f(x)-f\left(x+\frac yR\right)
  \right|^2
  e^{-|y|^2} \diff x \diff y \fstop
$$
For every $x,y$,
$$
  R^2
  \left|
    f(x)-f\left(x+\frac yR\right)
  \right|^2
  \le
  |y|^2
  \int_0^1
  \left|
    \nabla f\left(x+\frac{ty}{R}\right)
  \right|^2 \diff t \fstop
$$
After integration in $x$, the right-hand side is bounded by
$$
  |y|^2\int_{\C}|\nabla f(x)|^2 \diff x \comma
$$
which is integrable against $e^{-|y|^2}\diff y$.  Dominated
convergence therefore gives
\begin{align}
  \lim_{R\to\infty}\Var(F_R)
  &=
  \frac1{2\pi^2}
  \iint_{\C^2}
  |\nabla f(x)\cdot y|^2
  e^{-|y|^2} \diff x \diff y
  \notag\\
  &=
  \frac1{4\pi}
  \int_{\C}|\nabla f(x)|^2\diff x \comma
  \label{eq:LVL}
\end{align}
where we used
$
  \int_{\R^2}y_i y_j e^{-|y|^2} \diff y
  =
  \frac{\pi}{2}\delta_{ij}.
$
Combining \eqref{eq:SLE} and
\eqref{eq:LVL}, we obtain
$$
  \frac{\mathcal E(F_R)}
       {\Var(F_R)}
  \longrightarrow2 \fstop
$$
Hence the spectral gap is at most $2$.  Together with
\eqref{eq:SPG}, this proves that the sharp spectral gap
is exactly $2$.
\end{proof}

\paragraph{Disclosure} 
Interactive discussions with an AI tool assisted in exploring proof ideas and polishing the language.  All mathematical details were  independently verified by the author, who takes full responsibility for their correctness.

\bibliographystyle{alpha}
\bibliography{/Users/suzukikouhei/MasterBib.bib}

\end{document}